\documentclass[12pt]{amsart}

\numberwithin{equation}{section}
\usepackage{amsfonts}
\usepackage{amssymb}
\usepackage{amsmath}

\begin{document}

\title {Note on the best approximation in $L^1$ metric}
\author{ Alexey Solyanik }
\date{21 June 2016}
\dedicatory{To Yuriy Kryakin}
\address{Caribbean Sea, St. Marten, France}
\email{transbunker@gmail.com}
\newtheorem{remark}{Remark}
\newtheorem{theorem}{Theorem}
\newtheorem{lemma}{Lemma}
\newtheorem{proposition}{Proposition}
\newtheorem{corollary}{Corollary}
\newtheorem{example}{Example}
\maketitle
\vspace{1 cm}
\section{Introduction}

In this note we present one of the approaches to find the best (or good) approximation of the given function by trigonometric polynomials in $L^1$ metric.

The method is as follows. First we try  to represent the given function as the sum of Bernoulli kernels. For this purpose we use Fourier Analysis. Then we apply well-known Favard Theorem \cite{favard_1} concerning best approximation of Bernoulli kernels in $L^1$ metric.

This method is not new and actually goes back to the second paper of J. Favard \cite{favard_2} where decomposition of the given function as the infinite sum of shifted Bernoulli kernel (convolution) was applied for good (and some times the best) approximation of smooth function.

The new ingredient is that we decompose a given function not only to the sum of shifted \textit{one} Bernoulli kernel but also to the series  of \textit{all } even (or odd) Bernoulli kernels.

We remind that for $r=1,2, \dots $
\begin{equation}
\label{intr_ber_r}
\mathcal{B}_r(\theta) =\sum_{k\neq 0} \frac{e^{ik\theta}}{(ik)^r} 
\end{equation}
are the \textit{Bernoulli kernels}.

If $m\geq 2$ this series converge absolutely on $\mathbb{T}=\mathbb{R}/2\pi \mathbb{Z}$  and hence its sum is a continued function. For $m=1$ this series converge everywhere but to the discontinued function with a jump equal $2\pi$ at the point $0$.

We note that our definition of Bernoulli kernels is different by the constant factor $2$ with the usual one (see e.g. \cite{devore},  p. 150) since we use $2\pi$ instead of $\pi$ in the definitions of norm, convolution and Fourier coefficients. Thus for $f\in L^1$
$$
\|f\|_{L^1}=\frac{1}{2\pi}\int_0^{2\pi}|f(\varphi)|d\varphi
$$
and
$$
\widehat{f}(k)=\frac{1}{2\pi}\int_0^{2\pi}f(\varphi)e^{-ik\varphi}  d\varphi
$$ We refer to \cite{katz} for background theory of Fourier series and for  the notations.

In the seminal paper \cite{favard_1} Jean Favard found the exact approximation of Bernoulli kernels by the trigonometric polynomials in $L^1$ metric:
\begin{equation}
\label{intr_fav_best}
E_{n-1}(\mathcal{B}_r)_{L^1}= \frac{K_r}{n^r} 
\end{equation}
where  
\begin{equation}
\label{intr_fav_const}
K_r=\frac{4}{\pi}\sum_{k\in 4\mathbb{Z}+1} \frac{1}{k^{r+1}}
\end{equation}
and trigonometric polynomial $\tau_{n-1}^r (\theta)$ of best approximation has explicit formula.

Here and later in the sequel 
$$
E_m(f)_X=\inf_{t\in T_m} \| f-t\|_X
$$
where $X$ is $L^1$ or $L^\infty$ and $T_m=\lbrace t(\varphi) : t(\varphi)=\sum_{|k|\leq m }c_k e^{ik\varphi} \rbrace$ is the space of trigonometric polynomials (with complex coefficients $c_k$) of the degree less or equal $m$.

Favard Theorem (\ref{intr_fav_best})  was a corner stone of many approximation theorems (see e.g. \cite{favard_1}, \cite{favard_2},  \cite{favard_3}, \cite{nik}, \cite{tem}, \cite{kry} or chapter 7 of the book \cite{devore}). In \cite{nik} S. M. Nikolsky applied (\ref{intr_fav_best})  to obtain estimates of approximation of Lipschitz functions  by algebraic polynomials on the interval $[-1,1]$ and observed that these estimates depended on the position of the  point $x$ on the interval (see e.g. chapter 8 of \cite{devore} or Section \ref{_sec_alg_appr} of this note for the precise definitions).

The starting point of this article was a question posed to the author by Y. V. Kryakin -----  to improve constants in  the Nikolsky type inequality 
\begin{equation}
\label{intr_alg_nik}
|f(x)-P_{n}(f,x)| \leq  \frac{C_1}{n}
\sqrt{1-x^2}+ \frac{C_2}{n^2}\sqrt{x^2}
\end{equation}
 After standard substitution $x=\cos\theta$  this problem at once reduced to the trigonometric case and as usual, the main ingredient of the proof is good (or best) approximation of the corresponding kernels $\mathcal{K}_1(\theta)=\mathcal{B}_1(\theta)\cos\theta $ and  $\mathcal{K}_2(\theta)= \mathcal{B}_1(\theta)\sin\theta$ in $L^1$-metric by trigonometric polynomials.

Trigonometric polynomial interpolating kernel $\mathcal{K} $ at the equidistant points (nodes) is the desired polynomial of the best approximation. In order to prove this one have to show that the difference between kernel and interpolating polynomial change sign at every node and only at these nodes. Unlike Bernoulli kernels the  kernels $\mathcal{K}_1$ and $\mathcal{K}_2$ are not an algebraic polynomials, but quasi-polynomials. This fact not allow to apply standard arguments ---  quasi-polynomials does not vanishing after many differentiation.  

We gave three different proofs of this important property which at once implies the desired estimates (Theorem \ref{theo_quasi_bern}). First one is based on the modification of the \textit{Sturm sequence} in the way of N. Tschebotarow \cite{tsch} . Second one is based on the application of \textit{Sturmian arguments} of zero set analysis (see e.g. \cite{gal} p. 188).

Unfortunately both proofs was not too short.

The third proof, which is based on the decomposition of kernel to the series of Bernoulli kernels is the subject of present note and this approach seems to be new.

\section{Favard constants}
Constants $K_r$ in (\ref{intr_fav_best}) usually called the \textit{Favard constants} (see e.g. \cite{devore} , p.149)   and from (\ref{intr_fav_const}) $K_r=\frac{4}{\pi}S(r+1)$, where
$$
S(r)=\sum_{k\in 4\mathbb{Z}+1} \frac{1}{k^{r}}=1+\frac{(-1)^r}{3^r}+\frac{1}{5^r}+\frac{(-1)^r}{7^r}+\frac{1}{9^r}+\frac{(-1)^r}{11^r} + \cdots 
$$
Hence
\begin{equation}
\label{intr_fav_osc}
K_2<K_4<\cdots < \frac{4}{\pi} < \cdots K_3<K_1
\end{equation}

We also claim that  for odd $r$
\begin{equation}
\label{intr_fav_const_odd}
K_{r}=\frac{4}{\pi}\sum_{k=0}^\infty  \frac{1}{(2k+1)^{r+1}}
\end{equation}
and for even $r$
\begin{equation}
\label{intr_fav_const_even}
K_{r}=\frac{4}{\pi}\sum_{k=0}^\infty  \frac{(-1)^k}{(2k+1)^{r+1}}
\end{equation}
If we (following \cite{el}) compose the generating function for $K_r$ (where we define for convenience  $K_0=1$)
\begin{equation}
\label{intr_bern_gen}
K(z)=\sum_{r=0}^\infty K_rz^r=\sum_{r=0}^\infty\frac{4}{\pi} S(r+1)z^r=\frac{4}{\pi }\sum_{r=0}^\infty \sum_{m\in 4\mathbb{Z}+1}\frac{z^r}{m^{r+1}}  
\end{equation}
and then change order of sums, which is possible for $|z|<1$ we get
\begin{equation}
\label{intr_sum}
K(z)=\frac{4}{\pi }\sum_{m\in 4\mathbb{Z}+1}\frac{1}{m-z}=\frac{4}{\pi }(\frac{1}{1-z}-\frac{1}{3+z}+ \frac{1}{5-z}-\frac{1}{7+z}+\frac{1}{9-z}-\dots
\end{equation}
Thus, according to the well-known formulas  for $\tan z$ and $\sec z$ (see e. g. \cite{aar} formulas (1.2.7) and (1.2.8))
\begin{equation}
\label{fav_const_tan}
\frac{\pi}{2}\tan\frac{\pi  z}{2}=\frac{1}{1-z}- \frac{1}{1+z}+ \frac{1}{3-z}-\frac{1}{3+z}+ \frac{1}{5-z}-\frac{1}{5+z}+\dots 
\end{equation}
\begin{equation}
\label{fav_const_sec}
\frac{\pi}{2}\sec\frac{\pi  z}{2}=\frac{1}{1-z}+\frac{1}{1+z}- \frac{1}{3-z}-\frac{1}{3+z}+ \frac{1}{5-z}+\frac{1}{5+z}-\dots
\end{equation}
we have
\begin{equation}
\label{intr_final_formula_tan_cosec}
K(z)=\sum_{r=0}^\infty K_rz^r=\sum_{r=0}^\infty K_{2r+1}z^{2r+1}+\sum_{r=0}^\infty K_{2r} z^{2r}=\tan \frac{\pi z}{2}+\sec \frac{\pi z}{2}
\end{equation}
 Develop right side in Taylor series  near $z=0$ (\cite{as}, p. 75)
\begin{equation}
\label{intr_sec}
\sec z=1+ \frac{z^2}{2}+\frac{5z^4}{24}+\frac{61z^6}{720} +\dots = \sum_{n=0}^\infty  \frac{(-1)^n E_{2n}}{(2n)!}z^{2n}
\end{equation}
and 
\begin{equation}
\label{intr_tan}
\tan z= z+ \frac{z^3}{3}+\frac{2z^5}{15}+\frac{17z^7}{315}+ \dots=\sum_{n=1}^\infty  \frac{(-1)^{n-1}2^{2n}(2^{2n}-1)B_{2n}}{(2n)!}z^{2n-1}
\end{equation}
where $E_n$ and $B_n$ are Euler and Bernoulli  numbers respectively (\cite{as}, p. 804)
\begin{equation}
\label{intr_eul_num}
E_0= 1, ~~~ E_2=-1,~~~ E_4=5, ~~~
\end{equation}
\begin{equation}
\label{intr_eul_num}
B_0= 1, ~~~ B_1=-\frac{1}{2},~~~ B_2=\frac{1}{6},~~~B_4=-\frac{1}{30}
\end{equation}
Now comparing coefficients near $z^n$ in both sides of ({\ref{intr_final_formula_tan_cosec}) we can obtain exact values of $K_r$ 
\begin{equation}
\label{intr_first_favar}
K_1=\frac{\pi}{2},~~~ K_2=\frac{\pi^2}{8},~~~K_3=\frac{\pi^3}{24}, ~~~K_4=\frac{5\pi^4}{384},~~~K_5=\frac{\pi^5}{240},~~~K_6=\frac{61\pi^6}{46080},\dots 
\end{equation} 
Since coefficients in Taylor expansion of functions $\sec z$ and $\tan z$ near $z=0$ has only rational numbers, we can conclude that  all $K_r$ are $\pi^r$-rational.

\section{Best approximation of Steklov  kernels}

To explain the  method  we start with some well known result of the best approximation in $L^1$ metric (see \cite{kry} ). The only advantage of our approach  is that it is a little bit shorter. 

Define
$$
\chi_h(\theta)=h^{-1}~~~\text{for}~~~\vert \theta \vert \leq \pi h~~~ \text{and} ~~~0 ~~~\text{otherwise on}~~~\mathbb{T}
$$
and
$$
\chi^m_h(\theta)=(\chi_h\ast\chi_h\ast \dots \chi_h) (\theta)~~~ (m ~~~\text{times})
$$

Then it is easy to see that
\begin{equation}
\label{chi}
\chi^m_h(\theta)=1+(2\pi h )^{-m}\sum_{p=0}^m(-1)^{m-p} \binom{m} p  \mathcal{B}_m(\theta+(2p-m)\pi h)
\end{equation}
or
\begin{equation}
\label{chi_1}
\chi^m_h(\theta)=1+(2\pi h)^{-m}\Delta_{\pi h}^m \mathcal{B}_m(\theta)
\end{equation}
where $\Delta_h^m$ is the central difference.

In order to prove  (\ref{chi}) we claim that
\begin{equation}
\label{chi_f}
\widehat{\chi_h}(k)=\frac{\sin \pi hk}{\pi h k}
\end{equation}
and hence
\begin{equation}
\label{chi_f_m}
\widehat{\chi_h^m}(k)=\left(\frac{\sin \pi hk}{\pi h k}\right)^m=(2\pi h)^{-m} (e^{i\pi h k }- e^ {-i \pi h k})^m\widehat{\mathcal{B}_m}(k) 
\end{equation}
for $k\neq 0$. Hence for $k\neq 0$ the Fourier coefficients of the left side and of the right side of (\ref{chi}) are equal.

But for $k=0$ obviously $\widehat{\chi_h^m}(0)=1$ and for the  sum $\Sigma$  with Bernoulli kernel in the right side of (\ref{chi}) we always have  $\widehat{\Sigma}(0)=0$, since $\widehat{\mathcal{B}_m}(0)=0$.

Thus for all $k\in \mathbb{Z}$ Fourier coefficients of left and right sides of (\ref{chi}) coincides and hence the functions are equal.

The representation (\ref{chi})   and Favard identity (\ref{intr_fav_best}) immediately implies
\begin{equation}
\label{best_chi}
E_{n-1}(\chi_h^m)_{L^1}\leq \frac{K_m}{(\pi h n)^m}
\end{equation}
which is \textit{a good,} but not  \textit{the best} estimate for \textit{all combinations }of $h$ and $n$.

Meanwhile, for some combinations of $n$ and $h$ it is the best one and to show this we use duality.

The fundamental E. Helly (or Hahn-Banach) Theorem
\begin{equation}
\label{intr_dual_helly}
(X  /Y)^*\cong Y^\bot
\end{equation}
 implies that
\begin{equation}
\label{intr_dual_trig}
E_{n-1}(f)_{L^1}=\sup_{\sigma\in T^\bot_{n-1}, \|\sigma\|_\infty=1} \int_0^{2\pi} f (\varphi)\overline{ \sigma(\varphi)}d\varphi/ 2\pi 
\end{equation}

Here $T^\bot_{n-1}$ is ''annihilator space'' for $T_{n-1}$. i.e. space of functionals, which equal zero on each element from $T_{n-1}$, i. e.  they annihilate  polynomials.

Or in other words, if $\langle\cdot , \sigma \rangle$ is a functional  on the space $L^1$ generated by $\sigma\in L^\infty$, then  
$$
T^\bot_{n-1} =\lbrace \langle\cdot , \sigma \rangle \in (L^1)^* : \ker \langle\cdot , \sigma \rangle =T_{n-1}\rbrace
$$

Elementary Fourier analysis tell us that $\sigma\in T^\bot_{n-1}$  if and only if $\hat{\sigma}(k)=0$ for all $|k|\leq n-1$. Hence
\begin{equation}
\label{intr_dual_ber}
E_{n-1}(\chi_h^m)_{L^1}=\sup_{\sigma\in T^\bot_{n-1}, \|\sigma\|_\infty=1} \int_0^{2\pi}\chi_h^m(\varphi)\overline{ \sigma(\varphi)}d\varphi/ 2\pi 
\end{equation}
and from representation (\ref{chi})
\begin{equation}
\label{intr_dual_ber_TT}
E_{n-1}(\chi_h^m)_{L^1}=
\end{equation}
$$
(2\pi h)^{-m}\sup_{\sigma\in T^\bot_{n-1}, \|\sigma\|_\infty=1}    \sum_{p=0}^m(-1)^{m-p} \binom{m} p \int_0^{2\pi}          \mathcal{B}_m(\varphi)\overline{ \sigma(\varphi-(2p-m)\pi h)}d\varphi/ 2\pi 
$$
Define
\begin{equation}
\label{intr_dual_sin}
s_n(\varphi)=s(n\varphi)=\text{sgn} \sin (n \varphi)=\frac{2}{\pi }\sum_{k\in 2\mathbb{Z}+1} \frac{e^{ikn\varphi}}{ik}
\end{equation}
and
\begin{equation}
\label{intr_dual_cos}
c_n(\varphi)=c(n\varphi)=\text{sgn} \cos (n \varphi)= \text{sgn} \sin (n \varphi +\pi/2)=\frac{2}{\pi }\sum_{k\in 2\mathbb{Z}+1} (-1)^{\frac{k-1}{2}}\frac{e^{ikn\varphi}}{k}
\end{equation}

Then obviously $c_n( \varphi) \in T^\bot_{n-1}$ and has a unit $L^\infty$ norm. Remark, that since $\chi_h^m$ is an even function for all $m$,  we have to choose even $\sigma(\varphi)$ in (\ref{intr_dual_ber_TT})  to get estimate from below.

Let from now $m$ be an \textit{odd} number, let say $m=2r+1$ and $h=\frac{2k-1}{2n}$ for $k=1,2,\dots n$. Note, that for these values of $h$
$$
c_n(\varphi-(2p-m)\pi h)=(-1)^{p-r-k}\text{sign}\cos(n\varphi-\pi/2 )=(-1)^{p-r-k}s(n\varphi)
$$

Define $\sigma(\varphi)=(-1)^k c_n(\varphi )$.  Then from (\ref{intr_dual_ber_TT})
$$
E_{n-1}(\chi_h^m)_{L^1}\geq (-1)^k (2\pi h)^{-m}  \sum_{p=0}^m(-1)^{m-p+p-r-k} \binom{m} p \int_0^{2\pi}          \mathcal{B}_m(\varphi)\overline{s(n\varphi)}d\varphi/ 2\pi =
$$
$$
(-1)^{r}(\pi h)^{-m}\int_0^{2\pi}          \mathcal{B}_m(\varphi)\overline{s(n\varphi)}d\varphi/ 2\pi =(-1)^{r}(\pi h)^{-m}\frac{2i}{\pi }\sum_{k\in 2\mathbb{Z}+1} \frac{\widehat{\mathcal{B}_m}(kn)}{k}=
$$
$$
(-1)^{r}(\pi hn)^{-m}\frac{2i}{\pi }\sum_{k\in 2\mathbb{Z}+1} \frac{1}{k(ik)^{2r+1}}=(-1)^r(\pi hn)^{-m}\frac{2}{\pi }\sum_{k\in 2\mathbb{Z}+1} \frac{(-1)^r}{k^{2r+2}}=
$$
$$
(\pi hn)^{-m}\frac{4}{\pi }\sum_{k=0}^\infty \frac{1}{(2k+1)^{2r+2}}=\frac{K_m}{(\pi h n)^m}
$$
which is an opposite inequality to (\ref{best_chi}) and hence we have an equality for odd $m$ and  $h\in \mathcal{I}_n$, where
$$
\mathcal{I}_n=\left\lbrace  \frac{1}{2n}, \frac{3}{2n}, \frac{5}{2n}, \dots \frac{2n-1}{2n}\right\rbrace
$$

Let  now $m$ be an \textit{even} number, let say $m=2r$ and again $h\in \mathcal{I}_n$. Note, that for even $m$
$$
c_n(\varphi-(2p-m)\pi h)=(-1)^{p-r}\text{sign}\cos(n\varphi)=(-1)^{p-r}c(n\varphi)
$$

Hence, if we take $\sigma(\varphi)=c_n(\varphi)$ in (\ref{intr_dual_ber_TT}) we get  
$$
E_{n-1}(\chi_h^m)_{L^1}\geq (2\pi h)^{-m}  \sum_{p=0}^m(-1)^{m-p} \binom{m} p \int_0^{2\pi}          \mathcal{B}_m(\varphi)(-1)^{p-r}\overline{c(n\varphi)}d\varphi/ 2\pi =
$$
$$
(-1)^r(\pi h)^{-m}\int_0^{2\pi}          \mathcal{B}_m(\varphi)\overline{c(n\varphi)}d\varphi/ 2\pi =(-1)^r(\pi h)^{-m}\frac{2}{\pi }\sum_{k\in 2\mathbb{Z}+1} (-1)^{\frac{k-1}{2} }\frac{\widehat{\mathcal{B}_m}(kn)}{k}=
$$
$$
(\pi hn)^{-m}\frac{2}{\pi }\sum_{k\in 2\mathbb{Z}+1} \frac{(-1)^{\frac{k-1}{2}}}{k^{2r+1}}=(\pi hn)^{-m}\frac{4}{\pi}\sum_{k=0}^\infty  \frac{(-1)^k}{(2k+1)^{2r+1}}=\frac{K_m}{(\pi h n)^m}
$$
where we use  (\ref{intr_fav_const_even}), since  $m=2r$.    This is an opposite inequality to (\ref{intr_fav_best}) and hence we have an equality for even $m$ and  $h\in \mathcal{I}_n$.

We claim also that  $\chi_h^m(\varphi)=0$ for $|\varphi|>\pi mh$. Hence if $\pi/2n\geq \pi mh$  we can take    $\sigma(\varphi)=c_n(\varphi)$     to obtain
$$
\int_0^{2\pi}\chi_h^m(\varphi)\overline{ \sigma(\varphi)}d\varphi/ 2\pi=\int_0^{2\pi}\chi_h^m(\varphi)d\varphi/ 2\pi=1
$$
This observation implies
\begin{equation}
\label{best_chi_1}
E_{n-1}(\chi_h^m)_{L^1}= 1
\end{equation}
for all $h$ such that $|h|\leq \frac{1}{2mn}$.

If we compare (\ref{best_chi}) with (\ref{best_chi_1}) for $m=1$ we see that right side of  (\ref{best_chi}) equal $1$ for $h=1/2n$.

For $m\geq 2$ and $h=\frac{1}{2mn}$ right side of  (\ref{best_chi}) equal $(2m/ \pi)^m K_m > 1$ which indicate that our estimates are rude for these values of $h$.

Let we summarize all our observations together. Let numbers $n$ and $m$ are fixed. Then
\begin{equation}
\label{best_chi_first}
\text{for} ~~~ 0< h \leq\frac{1}{2mn} ~~~ \text{we have}~~~  E_{n-1}(\chi_h^m)_{L^1}= 1
\end{equation}
\begin{equation}
\label{best_chi_all}
\text{for} ~~~ \frac{1}{2mn} \leq  h \leq 1~~~ \text{we have}~~~  E_{n-1}(\chi_h^m)_{L^1}\leq \frac{K_m}{(\pi h n)^m}
\end{equation}
\begin{equation}
\label{best_chi_odd} 
\text{for} ~~~~h\in \left\lbrace  \frac{1}{2n}, \frac{3}{2n}, \frac{5}{2n}, \dots \frac{2n-1}{2n}\right\rbrace~~~ \text{we have}~~~  E_{n-1}(\chi_h^m)_{L^1}= \frac{K_m}{(\pi h n)^m}
\end{equation}

\section{Best approximation of quasi-Bernoulli kernels}  
Our second example are kernels $\mathcal{K}_1(\theta)=\mathcal{B}_1(\theta)\cos\theta$ and $\mathcal{K}_2(\theta)=\mathcal{B}_1(\theta)\sin\theta$, which play an important role in questions of  approximation of smooth functions by \textit{algebraic } polynomials on the segment $[-1,1]$. This is the subject of the last section and actually the starting point of this note. We call kernels $\mathcal{K}$ as well as more general kernels,  the \textit{quasi-Bernoulli }kernels since obvious analogy with quasi-polynomials.

\begin{theorem}
\label{theo_quasi_bern} Let $\mathcal{K}_1(\theta)=\mathcal{B}_1(\theta)\cos\theta$ and $\mathcal{K}_2(\theta)=\mathcal{B}_1(\theta)\sin\theta$. Then for all natural $n\geq 2$
\begin{equation}
\label{quasi_ber_ba_1}
E_{n-1} (\mathcal{K}_1)_{L^1}= \tan\frac{\pi}{2n}
\end{equation}
\begin{equation}
\label{quasi_ber_ba_2}
E_{n-1} (\mathcal{K}_2)_{L^1}= \sec\frac{\pi}{2n}-1
\end{equation}
\end{theorem}
\begin{proof}

To find the best approximation of these kernels by trigonometric polynomials we use formulas which seems unknown before --
\begin{equation}
\label{intr_golden_formula_1}
\mathcal{K}_1(\theta)=\frac{1}{2}\sin \theta+\sum_{r=0}^\infty((-1)^r\mathcal{B}_{2r+1}(\theta)-2\sin \theta))
\end{equation}
and 
\begin{equation}
\label{intr_golden_formula_2}
\mathcal{K}_2(\theta)=1+\frac{1}{2}\cos \theta + \sum_{r=1}^\infty((-1)^{r+1}\mathcal{B}_{2r}(\theta)+2\cos\theta)
\end{equation}

To prove these formulas one have only to compute  Fourier coefficients of both sides which  is a routine work (see also (\ref{kry_1_coef}), (\ref{kry_1_coef_1}), (\ref{kry_2_coef}), (\ref{kry_2_coef_1}))

Since, obviously for $r\geq 1$ we have  
$|(-1)^r \mathcal{B}_{2r+1}(\theta) -2\sin\theta | \leq 6 \cdot 2^{-(2r+1)}$ and  $|(-1)^{r +1}\mathcal{B}_{2r}(\theta) +2\cos \theta | \leq 6 \cdot 2^{-2r}$, these series converges absolutely and hence uniformly on $\mathbb{T}$.

The corresponding trigonometric polynomials of best approximation of $\mathcal{K}_1$ and  $\mathcal{K}_2$ in $L^1$ metric are
$$
\tau_{n-1}(\mathcal{K}_1)(\theta) =\frac{1}{2}\sin \theta+\sum_{r=0}^\infty((-1)^r\tau_{n-1}(\mathcal{B}_{2r+1})(\theta) -2\sin \theta)
$$
and
$$
\tau_{n-1}(\mathcal{K}_2)(\theta) =1+\frac{1}{2}\cos \theta+\sum_{r=1}^\infty((-1)^{r+1}\tau_{n-1}(\mathcal{B}_{2r})(\theta) +2\cos  \theta)
$$
where series converges absolutely, since 
$$
|(-1)^r\tau_{n-1}(\mathcal{B}_{2r+1})(\theta) -2\sin \theta|\leq E_{n-1}(\mathcal{B}_{2r+1})_{L^1}+|(-1)^r \mathcal{B}_{2r+1}(\theta) -2\sin\theta | \leq
$$
$$
\frac{K_{2r+1}}{n^{2r+1}}+ 6 \cdot 2^{-(2r+1)}
$$
and 
$$
|(-1)^{r+1}\tau_{n-1}(\mathcal{B}_{2r})(\theta) +2\cos \theta|\leq E_{n-1}(\mathcal{B}_{2r})_{L^1}+|(-1)^{r+1} \mathcal{B}_{2r}(\theta) +2\cos\theta | \leq
$$
$$
\frac{K_{2r}}{n^{2r}}+ 6 \cdot 2^{-2r}
$$

Now we immediately get the desired upper estimates for $E_{n-1}(\mathcal{K} )$  from  the Favard identities (\ref{intr_fav_best}) and formula (\ref{intr_final_formula_tan_cosec}).

In order to prove opposite inequality we again use duality. Define
\begin{equation}
e_n(\mathcal{K}_1)=\frac{1}{2\pi}\int_0^{2\pi} \mathcal{K}_1(\varphi)\text{sign}\sin n\varphi d\varphi  
\end{equation}
and
\begin{equation}
\label{kry_2_mat}
e_n(\mathcal{K}_2)=\frac{1}{2\pi}\int_0^{2\pi} \mathcal{K}_2(\varphi)(-\text{sign}\cos n\varphi) d\varphi 
\end{equation}

Then by the duality arguments we have $E_{n-1}(\mathcal{K} ) \geq e_n(\mathcal{K})$.

For $n\geq 2$ in the same manner as in the preceding section, we obtain
$$
e_n(\mathcal{K}_1)=\frac{2}{\pi}\sum_{k\in 2\mathbb{Z}+1} \frac{\widehat{\mathcal{K}_1}(-kn)}{ik}=\frac{2}{\pi}\sum_{k\in 2\mathbb{Z}+1} \frac{-kn}{ik((kn)^2-1)i}
=
$$
$$
\frac{2}{\pi}\sum_{k\in 2\mathbb{Z}+1} \frac{n}{((kn)^2-1)}=\frac{4}{\pi n}\sum_{k=0}^\infty  \frac{1}{((2k+1)^2-n^{-2})}
$$
In order to calculate the last sum we can use (\ref{fav_const_tan})

\begin{equation}
\label{kry_1_formula}
\sum_{k=0}^\infty \frac{1}{(2k+1)^2-z^2}=\frac{\pi}{4z}\tan\frac{\pi z}{2}
\end{equation}
Hence
\begin{equation}
\label{kry_1_finito}
e_n(\mathcal{K}_1)=\tan\frac{\pi}{2n}
\end{equation}
Remark, that
\begin{equation}
\label{kry_1_asympt}
e_n(\mathcal{K}_1)\asymp \frac{\pi}{2n}=\frac{K_1}{n} ~~~ \text{for} ~~~~~~ n\to \infty
\end{equation}
and 
\begin{equation}
\label{kry_2_asympt_l}
e_n(\mathcal{K}_1) >\frac{\pi}{2n} +\frac{\pi^3}{24n^3}=\frac{K_1}{n}+\frac{\pi^3}{24n^3}
\end{equation} 
where $K_1$ is the first Favard constant (see e.g. \cite{as}, p. 75).

Analogously for $n\geq 2$  
\begin{equation}
\label{kry_2_mat_calc}
e_n(\mathcal{K}_2)=\frac{2}{\pi}\sum_{k\in 2\mathbb{Z}+1}(-1)^\frac{k+1}{2} \frac{\widehat{\mathcal{K}_2}(-kn)}{k}=\frac{2}{\pi }\sum_{k\in 2\mathbb{Z}+1}(-1)^\frac{k-1}{2}  \frac{n^{-2}}{k(k^2-n^{-2})}
=
\end{equation}
$$
\frac{2}{\pi}\sum_{k\in 2\mathbb{Z}+1} (-1)^\frac{k-1}{2}(\frac{1/2}{k-n^{-1}}+\frac{1/2}{k+n^{-1}}-\frac{1}{k})=
$$
$$
=\frac{2}{\pi}(\frac{1}{1-n^{-1}}+\frac{1}{1+n^{-1}}-\frac{1}{3-n^{-1}}-\frac{1}{3+n^{-1}}+\frac{1}{5-n^{-1}}+\frac{1}{5+n^{-1}}-\dots -2\frac{\pi}{4})
$$
Now we used  (\ref{fav_const_sec}) to evaluate  the last sum and obtain
\begin{equation}
\label{kry_2_finito}
e_n(\mathcal{K}_2)=\sec\frac{\pi}{2n}-1
\end{equation}

Remark, that
\begin{equation}
\label{kry_1_asympt}
e_n(\mathcal{K}_2)\asymp \frac{\pi^2}{8n^2} =\frac{K_2}{n^2}~~~ \text{for} ~~~~~~ n\to \infty
\end{equation}
and 
\begin{equation}
\label{kry_2_asympt_l}
e_n(\mathcal{K}_2) >\frac{\pi^2}{8n^2} +\frac{5\pi^4}{384 n^4 }=\frac{K_2}{n^2}+\frac{5\pi^4}{384 n^4 }
\end{equation} 
where $K_2$ is the second Favard constant (see e.g. \cite{as}, p. 75).
\end{proof}

\section{Bernoulli series}
From previous two sections  we see that both representations has the form
\begin{equation}
\label{repr_gen}
\mathcal{K}=T+\sum_{r=1}^\infty \mathcal{B}_r\ast \mu_r
\end{equation}
with some trivial distribution $T$ ( trigonometric polynomial ) and non-trivial distributions  $\lbrace \mu_r \rbrace$. We remind that distributions is a continues linear functionals over $C^\infty (\mathbb{T)}$. The typical examples are measures (for instance the  Dirac measure $\langle f, \delta_\theta \rangle=f(\theta)$ with  $\widehat{\delta}(k)= e^{-ik\theta}$), derivative $\mu_r=\partial^r$ with $\widehat{\mu_r}=(ik)^r$ or absolutely continues measures $\mu_r=g_r(\theta)d\theta$.  See ( \cite{ed}  p. 52) for definitions and explanations.

Indeed, representation (\ref{chi}) has the form (\ref{repr_gen})  with $T_0(\theta)=1$ and   all $\mu_r$ equal $0$ except $r=m$, where
$$
\mu_m=(2 \pi h )^{-m} (\delta_{\pi h} -\delta_{-\pi h})^m
$$

On the other hand representation (\ref{intr_golden_formula_1})  has the form (\ref{repr_gen}) with $T(\theta)=\frac{1}{2}\sin \theta$ and $\mu_{2r}=0$ and $\mu_{2r+1}=(-1)^r\delta_0 -2\sin \theta d\theta$.

We remark that if we choose $\mu_r=\partial^r$ and $\mu_k=0$ for $k\neq r$ then the representation (\ref{repr_gen}) (as a functional over space $W^r$ ) is the familiar Euler-Maclauren formula (in the periodic case) with $\mathcal{K}=\delta_0$ and $\langle f, T\rangle =\widehat{f}(0)$. 

We also claim,  that any function $\mathcal{K}\in L^1$ has representation (\ref{repr_gen}) if we choose $\mu_r(\theta)=\widehat{K}(-r)(-ir)^re^{-ir\theta}+\widehat{K}(r)(ir)^re^{ir\theta}$ and $T(\theta)=\widehat{\mathcal{K}}(0)$, but we mainly interested in such representations where distributions are  measures with non trivial discrete part.
In this section we show that expansion (\ref{repr_gen}) with  measures with non trivial discrete part hold for quite general class of functions.

Next proposition in this direction corresponds to the  expansion (\ref{repr_gen}) with $\mu_r=c_r(\delta-D_N(\theta)d\theta)$, where $\delta$ is the Dirac measure and $D_N (\theta)$ is the Dirichlet kernel.
\begin{lemma}
\label{gen_g}
Let $\mathcal{K}(\theta)\in L^1(\mathbb{T})$ and there exist natural number $N$ such that for $k\in \mathbb{Z}$ and $|k|\geq N+1$ 
\begin{equation}
\label{gen_g_f}
\widehat{\mathcal{K}}(k)=g(k)
\end{equation}
where $g(z)$ is an analytic function in $\lbrace z: ~~|z|>N\rbrace$ and has a zero at  infinity. 

Then
 \begin{equation}
\label{gen_g_b}
\mathcal{K}(\theta)=T_N(\theta) + \sum_{m=1}^\infty c_m( \mathcal{B}_m(\theta)-S_N(\mathcal{B}_m,\theta))
\end{equation}
where
 \begin{equation}
\label{gen_g_co}
c_m=\frac{1}{2\pi i}\int_{C} (i\zeta)^mg(\zeta)\frac{d\zeta}{\zeta}
\end{equation}
and integral is over  some  contour $C$ around closed  disc $\lbrace z : |z|\leq N\rbrace$.

  Here $S_N(\mathcal{B}_m ,\theta)$ is the $N$-th partial sum of Fourier series of the Bernoulli kernel. i. e.
  $$
  S_N(\mathcal{B}_{2r+1} ,\theta)=(-1)^{r}2\sum_{k=1}^N \frac{\sin k \theta}{k^{2r+1}} 
  $$
  $$
  S_N(\mathcal{B}_{2r} ,\theta)=(-1)^{r}2\sum_{k=1}^N \frac{\cos k \theta}{k^{2r}}
  $$
  and 
  $$
  T_N(\theta)=\sum_{|k|\leq N} \widehat{\mathcal{K}}(k)e^{ik\theta}
  $$
\end{lemma}

\begin{proof}
Since $g(z)$ is analytic outside of the closed  disc $\lbrace z: |z|\leq  N \rbrace$ and has a zero at infinity, it can be expanded  in $\lbrace z : |z|> N\rbrace $   into the Laurent series
$$
g(z)=\sum_{m=1}^\infty \frac{c_m}{(iz)^m}
$$
with $c_m$ defined by (\ref{gen_g_co}) (see e.g. \cite{ww}, p. 100). 

Denote $\mathcal{S}(\theta)$ the right side of  (\ref{gen_g_b})  and  claim that $|\mathcal{B}_m(\theta)-S_N(\mathcal{B}_m, \theta)|\leq 6\cdot (N+1)^{-m}$ on $\mathbb{T}$ and from (\ref{gen_g_co})  follows that $|c_m|\leq M (N+0.5)^m$. Hence series (\ref{gen_g_b}) converge absolutely and we can change sum and integral. This implies that for $|k|\leq N$
\begin{equation}
\label{gen_g_proof_k_f_N}
\widehat{\mathcal{S}}(k)=\widehat{T_N}(k)=\widehat{\mathcal{K}}(k)
\end{equation}
and for $|k|> N$
\begin{equation}
\label{gen_g_proof_k_f}
\widehat{\mathcal{S}}(k)=\sum_{m=1}^\infty c_m \widehat{\mathcal{B}_m} (k)=\sum_{m=1}^\infty  \frac{c_m}{(ik)^m}=g(k)=\widehat{\mathcal{K}}(k)
\end{equation}
Thus all Fourier coefficients of $\mathcal{S}(\theta)$ and $\mathcal{K}(\theta)$ coincides and hence they are equal.
\end{proof}

Now we can apply this lemma to kernels, which arises in the questions of approximations of functions with bounded  derivative of high order by algebraic polynomials on the segment $[-1,1]$.

Let 
\begin{equation}
\label{bern_ser_quasi_ber}
\mathcal{K}(\theta)=\mathcal{B}_1 (\theta) t_1(\theta)+\mathcal{B}_2(\theta)  t_2(\theta)+ \dots \mathcal{B}_m(\theta) t_m(\theta)
\end{equation}
where $\mathcal{B}_k(\theta)  $ are Bernoulli kernels  and $t_k(\theta)$ are (complex) trigonometric polynomials. Then we call $\mathcal{K}(\theta)$ the \textit{quasi-Bernoulli} kernel

For example $\mathcal{K}_1(\theta)=\mathcal{B}_1(\theta) \cos \theta$ and $\mathcal{K}_2(\theta)=\mathcal{B}_1(\theta) \sin \theta$ are quasi-Bernoulli kernels. We see that  for $k\neq\pm 1$
\begin{equation}
\label{kry_1_coef}
\widehat{\mathcal{K}}_1(k) =\frac{1}{2}(\mathcal{B}_1(\theta) e^{i\theta}+\mathcal{B}_1(\theta)e^{-i \theta})\hat{}(k)=\frac{1}{2}(\widehat{\mathcal{B}_1}(k-1)+\widehat{\mathcal{B}_1}(k+1))=\frac{k}{(k^2-1) i}
\end{equation}
and
\begin{equation}
\label{kry_1_coef_1}
\widehat{\mathcal{K}}_1(\pm 1) =\pm \frac{1}{4i}, ~~~~~~\widehat{\mathcal{K}}_1(0) =0
\end{equation}

Analogously for $k\neq\pm 1$
\begin{equation}
\label{kry_2_coef}
\widehat{\mathcal{K}}_2(k) =\frac{1}{2i}(\mathcal{B}_1(\theta) e^{i\theta}-\mathcal{B}_1(\theta)e^{-i \theta})\hat{}(k)=\frac{1}{2i}(\widehat{\mathcal{B}_1}(k-1)-\widehat{\mathcal{B}_1}(k+1))=-\frac{1}{(k^2-1) }
\end{equation}
and
\begin{equation}
\label{kry_2_coef_1}
\widehat{\mathcal{K}}_2(\pm 1) = \frac{1}{4}, ~~~~~\widehat{\mathcal{K}}_2(0) =1
\end{equation}

Hence $\widehat{\mathcal{K}}_1(k) =g(k)$ for $|k|>1$ with
$$
g(z)=\frac{z}{i(z^2-1)}
$$
The function $g(z)$ is analytic in $|z|>1$ and obviously have a zero at infinity. Thus we can apply Lemma \ref{gen_g}. To calculate the coefficients we can use the residue theorem (\cite{ww}, p. 112) and get 
 \begin{equation}
\label{gen_g_co_k_1}
c_m=\frac{i^m}{2\pi i}\int_{C} \frac{\zeta^{m}}{i(\zeta^2-1)}d\zeta=\frac{i^{m-1}}{2\pi i}\int_{C} \zeta^{m}\frac{1}{2}( \frac{1}{\zeta -1} -  \frac{1}{\zeta +1} )d\zeta=
\end{equation}
$$
i^{m-1}\frac{1}{2}(1-(-1)^m)
$$
Hence for $r=0, 1, 2, \dots$
 \begin{equation}
\label{gen_g_co_k_1_2r}
c_{2r}=0 ~~~~~\text{and}~~~~~c_{2r+1}=(-1)^{r}
\end{equation}

We see that according to Lemma \ref{gen_g}
$$
\mathcal{K}_1(\theta)=\sum_{|k|\leq 1} \widehat{\mathcal{K}}(k)e^{ik\theta}+\sum_{r=0}^\infty c_{2r+1}( \mathcal{B}_{2r+1}(\theta)-S_1(\mathcal{B}_{2r+1},\theta))=
$$
$$
\frac{1}{2}\sin \theta +\sum_{r=0}^\infty (-1)^r( \mathcal{B}_{2r+1}(\theta)-(-1)^{r}2\sin\theta)
$$
which coincide with (\ref{intr_golden_formula_1})

Analogously we can obtain (\ref{intr_golden_formula_2}). Moreover we can use the same approach to calculate coefficients in representation for any quasi-Bernoulli kernel defined by the formula (\ref{bern_ser_quasi_ber}), since $g(z)$ in this case  is always rational function and has all poles inside of the disc with radius bigger that maximum degree of polynomials $t_m(\theta)$.

\section{Approximation of Lipschitz functions by algebraic polynomials}
\label{_sec_alg_appr}

In \cite{favard_3} J. Favard   posed and partially solved the problem of the best approximation of Lipschitz's functions by algebraic polynomials on the segment $\mathbb{I}=[-1,1]$.

Denote $W^1=W^1(\mathbb{I})$ the space of all Lipschitz functions, or in other words absolutely continues functions with finite norm  $\|f\|_{W^1}=\|f'\|_{L^\infty}$. We prefer use the same symbols for the algebraic case, since  it is clear from the context of the exposition which case we consider. Thus we define
$$
E_m(f)_X=\inf_{p\in P_m} \| f-p\|_X
$$
where $X$ is $L^1(\mathbb{I})$ or $L^\infty(\mathbb{I} ) $ and $P_m=\lbrace p(x) : p(x)=\sum_{k=0}^m p_k x^k \rbrace$ is the space of algebraic polynomials  of the degree less or equal $m$. We also define 
$$
E_m(X)=\sup_{f: \| f\|_X  \leq 1} E_m(f)_X
$$
The result of J. Favard can be expressed in the form
$$
\frac{1}{n}<E_{n-1}(W^1)<\frac{K_1}{n}
$$  
This result was improved in \cite{nik} by S. M. Nikolsky .

\begin{equation}
\label{alg_nik}
 E_{n-1}(W^1)= \frac{K_1}{n}-\epsilon_n
\end{equation}
with
$$
\epsilon_n>0,~~~~~~ \epsilon_n=O(\frac{1}{n\log n})
$$
The exact value (formula) for $E_n(W^1)$ seems still unknown.

 Since algebraic case reducing to the trigonometric case by the substitution $x=\cos \varphi$, the main questions concentrate around  best constants in the inequalities and how constants depends of the point position on the interval. This dependence  was first observed by  S. M. Nikolsky in the same paper, were was noticed  that for the  algebraic Favard means $U_{n-1}(f,x)$ which are an algebraic polynomial of degree less or equal $n-1$
 \begin{equation}
\label{intr_alg_nik_log}
|f(x)-U_{n-1}(f,x)| \leq \frac{K_1}{n}\sqrt{1-x^2}+ C\frac{\log n}{n^2}\sqrt{x^2}
\end{equation}
for all functions $f\in W^1$. Here $K_1$ is the first Favard constant. This is immediately corollary of Favard Theorem in trigonometric case, but S. M. Nikolsky also proved that for \textit{Favard means} $\log n$ factor can not be removed.

The last result indicate that in the algebraic case,  instead of trigonometric, the constant near the term $1/n$ depends on the position of the point $x$ on the interval $[-1,1]$ and the reminder term (if exist) disappear  near point $0$  and near point $1$ dominate.

Then many papers was dedicated to remove  $\log$ factor using some linear methods of approximation by algebraic polynomials and to generalise Nikolsky's type inequality to high derivatives. The factor was removed and theorems was generalised, but constants, which appeared in the proofs was given with no fine estimates. 

 V. N. Temlyakov in \cite{tem} improved constants in the Nikolsky's type inequality (\ref{intr_alg_nik_log})  but, of course, for different approximation polynomials and, of course, with a little bigger  factor near $\sqrt{1-x^2}$. But the factor near $\sqrt{x^2}$ (reminder term) have fine and simple expression.

\begin{equation}
\label{alg_tem}
|f(x)-P_{n}(f,x)| \leq  \frac{K_1}{n}
\sqrt{1-x^2}+ \frac{2 K_2}{n^2}\sqrt{x^2}
\end{equation}

Here $K_1$ and $K_2$ are first and second Favard constants.

In this note we found, in some sense, optimal factors near $\sqrt{1-x^2} $ and $\sqrt{x^2}$ in  (\ref{alg_tem}).

\begin{theorem}
\label{theo_main}
Let  $f(x)$ be a function defined on $\mathbb{I}$ and $\|f\|_{W^1}\leq 1$.

Then for every $n\geq 2$ there exist an algebraic polynomial $P_{n}(f,x)$ of degree less or equal $n$, such that
\begin{equation}
\label{alg_est_t_s}
|f(x)-P_{n}(f,x)| \leq T(n)   
\sqrt{1-x^2}+ S(n) \sqrt{x^2}
\end{equation}
where 
$$
T(n)=\tan\frac{\pi}{2n}
$$ 
and 
$$
S(n)=\sec\frac{\pi}{2n}-1
$$.

\end{theorem}

Before proof some remarks are in order. 

Asymptotically for big $n$
\begin{equation}
\label{intr_tem_T}
T(n)=\frac{K_1}{n}+ O(n^{-3})
\end{equation}
and 
\begin{equation}
\label{intr_tem_S}
 S(n)=\frac{K_2}{n^2}+O(n^{-4})
\end{equation}
Thus, factor near $\sqrt{1-x^2}$ is asymptotically the same as in Temlyakov's Theorem, but the  factor near $\sqrt{x^2}$ is $2$ times better.

We also note that the first factor $T(n)$ can not be equal or less than factor in Temlyakov's Theorem with the optimal second factor.

These estimate are sharp in the sense that factors near $\sqrt{1-x^2}$ and $\sqrt{x^2}$ can not be improved \textit{simultaneously} too much. The exact sense of the last remark will be a subject of forthcoming note.
\begin{proof}

First we transform  the question to the trigonometric case by substitution $x=\cos \theta$. This is a standard approach since \cite{favard_3} and \cite{nik} (see also \cite{bust}). After this the original problem of approximation by algebraic polynomials looks as follows.

Let function $g\in N^1(\mathbb{T})$, i. e. $g(\theta)$ is periodic with period $2\pi$, even, absolutely continues function, such that $g'(\theta)=h(\theta)\sin \theta$, where $h(\theta)\in L^\infty$. Define $\|g\|_{N^1}=\|h\|_{L^\infty}$.   We shall approximate $g(\theta)$ by trigonometric polynomials of degree less or equal $n-1$ and since in $N^1(\mathbb{T})$ the norm is not  translation invariant,  we expect that the desired estimate have to depend of the position of the point $\theta$ on $\mathbb{T}$.

As usual, we use  the representation (see e.g. \cite{devore}, p. 211 and \cite{tem})
\begin{equation}
\label{ker_repr}
g(\theta)=\int_0^{2\pi} \mathcal{B}_1(\theta -\varphi) h(\varphi)\sin\varphi d\varphi/2\pi=\sin \theta (h\ast \mathcal{K}_1)(\theta)-\cos \theta (h\ast \mathcal{K}_2)(\theta)
\end{equation}
where
\begin{equation}
\label{kry_1}
\mathcal{K}_1(\theta)=\mathcal{B}_1(\theta)\cos\theta
\end{equation}
and
\begin{equation}
\label{kry_1}
\mathcal{K}_2(\theta)=\mathcal{B}_1(\theta)\sin \theta
\end{equation}

Let $\tau_{n-1}(\mathcal{K}_1)$ and $\tau_{n-1}(\mathcal{K}_2)$ be a trigonometric polynomials of best approximation of $\mathcal{K}_1$ and $\mathcal{K}_2$ in $L^1$ metric.

Then 
\begin{equation}
\label{ker_repr}
g(\theta)-(h\ast \tau_{n-1}(\mathcal{K}_1))(\theta) \sin\theta + (h\ast \tau_{n-1}(\mathcal{K}_2)(\theta)\cos\theta =
\end{equation}
$$
= h\ast (\mathcal{K}_1-\tau_{n-1}(\mathcal{K}_1))(\theta)\sin \theta -h\ast (\mathcal{K}_2-\tau_{n-1}(\mathcal{K}_2))(\theta)\cos \theta 
$$
and hence
\begin{equation}
\label{ker_est}
| g(\theta)-t_n(g)(\theta)|\leq E_{n-1} (\mathcal{K}_1) |\sin\theta| +E_{n-1} (\mathcal{K}_2) |\cos\theta|
\end{equation}
for trigonometric polynomial
\begin{equation}
\label{ker_pol}
t_n(g)(\theta)=(h\ast \tau_{n-1}(\mathcal{K}_1))(\theta) \sin\theta - (h\ast \tau_{n-1}(\mathcal{K}_2)(\theta)\cos\theta 
\end{equation}
of the degree $n$.
This is standard arguments. The problem is how to find polynomials of best approximations of $\mathcal{K}$ kernels and how to calculate the values of these best approximations.

We apply Theorem \ref{theo_quasi_bern}, which implies that
\begin{equation}
\label{ker_est}
| g(\theta)-t_n(g)(\theta)|\leq T(n)|\sin\theta| +S(n)|\cos\theta|
\end{equation}
and after substitution  $x=\cos \theta$ we get (\ref{alg_est_t_s}) where $P_n(f,x)$ is a corresponding algebraic polynomial in the Tschebicheff basis.

\end{proof}

\begin{remark} In the book \cite{bust} Temlyakov's estimate (\ref{alg_tem}) (Theorem 2.3.2 on the page 14) printed with the wrong factor near $\sqrt{x^2}$ --- instead $2K_2/n^2$ author put $K_2/n^2$, probably since of misprint. But we shall to stress out, that estimate with factor $K_2/n^2$ near $\sqrt{x^2}$ is actually wrong. This will be a subject of forthcoming note.

\end{remark}


\begin{thebibliography}{XXXX}

\bibitem[AS 1972]{as} Milton Abramovitz and  Irene Stegun (ed.) , \textit{Handbook of Mathematical Functions}, \ \ National Bureau of Standards,\ Appl. Math. Ser., v. 55 , \ 10-th print. with corr. , \ 1972

\bibitem[AAR 1999]{aar} George E. Andrews, Richard Askey, Ranjan Roy,   \ \textit{Special functions},\ Encyclop. Math. Appl. \ v. 71 \ Cambridge Univ. Press. \  1999

\bibitem[B 2012]{bust} Jorge Bustamante,   \ \textit{Algebraic Approximation: A Guide to Past and Current Solutions}, \ Birkhauser,\    Springer\ 2012


\bibitem[E 1982]{ed} R. E. Edwards,   \ \textit{Fourier series, a modern introduction}, \ v. 2, \ Second Ed.,\    Springer-Verlag\ N. Y.,\ 1982


\bibitem[El 2003]{el} Noam  D. Elkies, \ \textit{On the sums $\sum_{k=-\infty}^\infty  (4k+1)^{-n}$ }, \ arXiv:math/0101168v5 [math.CA] 5 Aug 2003

\bibitem[F 1936] {favard_1}   Jean Favard,   \ \textit{Application de la formule summatoire d'Euler a la demonstration de quelques proprietes extremales des integrales des fonctions periodiques et presque-periodiques,} \ Math. Tidsskr.,B, 1936, Copenhague. \ p. 81-94

\bibitem[F 1937] {favard_2}   Jean Favard,   \ \textit{Sur les meilleurs procedes d'approximation de certaines classes de fonctions par des polynomes trigonometriques}, \ Bull. des Sciences  math.\ v. 61, (1937), \ Juliet-August, \  pp.  209-224, 243-256

\bibitem[F 1938] {favard_3} Jean Favard, \ Sur l'approximation des fonctions, \ Bull, de Sc. Math., \ 61
\ (1938) \ 1, pp. 338-351

\bibitem[GH 2005] {gal} Victor A. Galaktionov, Petra J. Harwin,\ \textit{Sturm's Theorems of Zero Sets in Nonlinear Parabolic Equations}, \ in "Sturm-Liouville Theory: Past and Present", \ ed. W. O. Amrein, A. M. Hinz, D. B. Pearson\ Springer-Verlag\ 2005, \ p. 173-199

\bibitem[K 2002] {katz} Itzhak Katznelson,\ \textit{An Introduction to Harmonic Analysis}, \ 3 corr. ed., \ Dover Publications ,\ N.Y.,\ 2002

\bibitem[BK 2008]{kry} A. G. Babenko, Y. V. Kryakin, \ \textit{$L$ -approximation of $B$-splines by trigonometric polynomials}, \ arXiv:math/0811.0686v1 [math.CA] 5 Nov  2008

\bibitem[N 1946] {nik} S. M. Nikolsky, \ \textit{On the best approximation of functions satisfying Lipshitz's conditions by polynomials}, \ Izv. Akad. Nauk SSSR Ser. Mat. , \ 10:4 (1946), \ 295-322

\bibitem[T 1981] {tem} Vladimir N. Temlyakov, \ \textit{Approximation of functions from the Lipschitz class by algebraic polynomials} ,\ Math. Notes,\ 29:4 (1981), \  306-309

\bibitem[Ts 1942] {tsch} N. Tschebotarow, \ \textit{On some modification of Sturm and Fourier methods} ,\ DAN,\ \textbf{34}:(1942), 3-6 \ in '' Sobranie sochineniy, tom 2 '',  \ 1949,\ 109-112 (in Russian) 
 

\bibitem[VL 1993]{devore} Ronald de Vore, George G. Lorentz ,\ \textit{Constructive Approximation},\ Springer-Verlag,\  1993

\bibitem[WW 1950]{ww} E. T. Whittaker, G. N. Watson,  \ \textit{A course of modern analysis}, \ Fourth Ed.\ Cambridge University Press \ 1950






\end{thebibliography}
\end{document}